\documentclass{amsart}
\usepackage{amsfonts}
\usepackage{amsmath,amsthm}
\usepackage{amsfonts,amssymb}

\usepackage{enumerate}

\newtheorem{thm}{Theorem}[section]

\newtheorem{lem}[thm]{Lemma}

\newtheorem{rem}[thm]{Remark}

\numberwithin{equation}{section}

\newcommand{\al}{\alpha}

\def\lz{\lambda}
\def\Lz{\Lambda}

\def\sz{\sigma}

\def\({\Bigl(}
\def \){ \Bigr)}

 \def\RR{{\mathbb R}}

 \def\supp{\operatorname{supp}}

\def\sz{\sigma}

\def\ss{{\Bbb S}^{d-1}}
\def\bd{\mathbb{B}^d}

\begin{document}
\def\RR{\mathbb{R}}
\def\Exp{\text{Exp}}
\def\FF{\mathcal{F}_\al}

\title[] {Entropy numbers of weighted Sobolev classes on the unit sphere }

\begin{abstract}
We obtain the asymptotic orders of entropy numbers of Sobolev classes on the unit sphere with Dunkl weight which associates with the general finite reflection
group. Moreover, the
asymptotic order of entropy numbers of weighted Sobolev classes on the unit ball and on the standard simplex are discussed.
\end{abstract}
\author{Heping Wang} \address{ School of Mathematical Sciences, Capital Normal
University,
Beijing 100048,
 China.}
\email{wanghp@cnu.edu.cn.}
\author{Kai Wang}
\address{ School of Science, Langfang Normal University,
Langfang 065000, China.} \email{cnuwangk@163.com.}

\thanks{
 The first author was partially supported by the National Natural Science Foundation of China (11671271), and the Beijing Natural Science Foundation (1172004). The second author was partially supported by the National Natural Science Foundation of China (11801245), Natural Science Foundation of Hebei Province (A2018408044), and the Foundation of Education Department of Hebei Province (QN2017127). }


\maketitle
\input amssym.def

\section{Introduction}

The purpose of this paper is to study the entropy numbers of weighted Sobolev
space on the sphere. Let $\mathbb{S}^{d-1}:=\{x:\|x\|=1\}$
denote the unit sphere in $\mathbb{R}^d$ endowed with the rotation invariant
measure $d\sz$ normalized by $\int_{\mathbb{S}^{d-1}}d\sigma{(x)}=1$, where
$\|\cdot\|$ denotes the usual Euclidean norm.

Given a nonzero vector $u$, the
reflection $\sigma_u$ with respect to the hyperplane perpendicular to $u$ is defined
by $x\sigma_u:=x-\frac{2\langle x,u\rangle}{\langle u,u\rangle}u,x\in \mathbb{R}^d$, where $\langle\cdot,\cdot\rangle$
denotes the usual Eulidean inner product. A root system $\mathcal{R}$ is a finite
subset of nonzero vectors in $\mathbb{R}^d$ such that $u,v\in \mathcal{R}$ implies
$u\sigma_v\in \mathcal{R}$. For a fixed $u_0\in\mathbb{R}^d$ such that $\langle u,u_0\rangle\neq0$ for
all $u\in \mathcal{R}$, the set of positive roots $\mathcal{R}_+$ with respect to $u_0$ is
defined by $\mathcal{R}_+=\{u\in \mathcal{R}:\langle u,u_0\rangle>0\}$ and $\mathcal{R}=\mathcal{R}_+\cup(-\mathcal{R}_+)$.
A finite reflection group G with root system $\mathcal{R}$ is
a subgroup of orthogonal group $O(d)$ generated by $\{\sigma_u:u\in \mathcal{R}\}.$

Now we define a real function on $\mathcal{R}_+$, which is called multiplicity function,
written as $\kappa_v:v\mapsto\kappa_v$ of $\mathcal{R}_+\mapsto \mathbb{R}$ satisfying the property
that $\kappa_u=\kappa_v$ whenever $\sigma_u$ is conjugate to $\sigma_v$ in $G$,
that is, there is a $w$ in the reflection group $G$ generated by $\mathcal{R}_+$, such that $\sigma_uw=\sigma_v$.
From the definition of multiplicity function, we can see that $\kappa_v$ is a $G$-invariant
function.

In the following, we will consider the Dunkl weight function
$$h_\kappa(x)=\prod\limits_{v\in \mathcal{R}_+}|\langle x,v\rangle|^{\kappa_v},x\in\mathbb{R}^d,\ \ \kappa_v\ge0,$$
the function $h_\kappa$ is a homogeneous function of degree $\gamma_\kappa=\sum_{v\in \mathcal{R}_+}\kappa_v$
and invariant under the group $G$.
For simplicity of notation, we denote by
$$h_\kappa(x):=\prod\limits_{i=1}^{\#\mathcal{R}_+}|\langle x,v_j\rangle|^{\kappa_j}:=\prod\limits_{i=1}^{m}|\langle x,v_j\rangle|^{\kappa_j},$$
where $\# E$ is the number of the elements in $E$.
The simplest example of $h_\kappa$ is
$$h_\kappa(x)=\prod\limits_{i=1}^d|x_i|^{\kappa_i}\ \  \kappa_i\geq0,$$
corresponding to the group $G=\mathbf{Z}_2^d$.

Let $1\le p\le\infty,$ we denote by $L_p(h_\kappa^2)$ the usual weighted
Lebesgue space on $\ss$ with the finite norm
$$\|f\|_{p,\kappa}:=\bigg(\frac{1}{a_d^\kappa}\int_{\ss}|f(x)|^p h_\kappa^2(x)d\sigma(x)\bigg)^{1/p}<\infty,$$
where $a_d^\kappa=\int_{\ss}h_\kappa^2(x)d\sigma(x)$ is the normalization constant.
For $p=\infty$, we assume that $L_\infty$ is replaced by $C(\ss)$, the space of continuous function
on $\ss$ with the usual norm $\|\cdot\|_\infty$.
It is worthwhile to point out that
the Dunkl weight plays an important role in the theory of multivariate
orthogonal polynomials and will be of great use for the proofs of our results.

Let $K$ be a compact subset of a Banach space $X$. The $n$th entropy
number $e_n(K,X)$ is defined as the infimum of all positive
$\varepsilon$
 such that there exist $x_1,\dots, x_{2^{n}}$ in $X$ satisfying
 $ K\subset\bigcup_{k=1}^{2^{n}}(x_k+\varepsilon B_X),$ where $B_X$ is the unit ball of $X$, that is,
 $$e_n(K,X)=\inf\{\varepsilon>0: K\subset\bigcup_{k=1}^{2^{n}}(x_k+\varepsilon B_X),\  x_1,\dots, x_{2^{n}}\in X\}.$$
Let $T\in L(X,Y)$ be a bounded linear operator between the Banach
spaces $X$ and $Y$. The entropy number $e_n(T)$ is defined as
$$e_n(T):=e_n(T:X\mapsto Y)=e_n(T(B_X),Y).$$
Entropy number plays important roles in many related fields including the function space theory(\cite{ET}), $m$-term approximation(\cite{HW0,Te}),
and so on. In recent years, this classical approximation characterization draws more application in information-based complexity and tractalility problem (\cite{H,KMU,MUV}), signal and image processing (\cite{CDGO, Do}), learning theory (\cite{CS,KT0}).
In this paper, we shall determine asymptotic orders of entropy numbers of
Sobolev classes $BW_p^r(h_\kappa^2)$ in
$L_q(h_\kappa^2)$ for all $1\le p, q\le \infty$ (see the definitions of $W_p^r(h_\kappa^2)$ in Section 2).
In the unweighted case, the exact orders of the entropy numbers of Sobolev
classed $BW^r_p$ on the sphere in $L_q$ were obtained by Kushpel and Tozoni (\cite{KT})
for $1 < p,q < \infty$ and H. Wang, K. Wang and J. Wang (\cite{WWW}) for the remaining
case (when $p$ and/or $q$ is equal to $1$ or $\infty$). In the special case $G=\mathbf{Z}_2^d$ , the Kolmogorov, linear, and Gelfand widths of the
weighted Sobolev classes on the sphere in weighted $L_q$ space were obtained by
Huang and Wang (see \cite{HW}). Our main result  can be formulated as follows:

\begin{thm} Let $r>(d-1)(\frac{1}{p}-\frac{1}{q})_{+}(2\gamma_\kappa+1)$,
$1\leq p,q\leq\infty$, then we have
\begin{equation*}e_n\big(BW_{p}^r(h_\kappa^2),\
L_q(h_\kappa^2)\big) \asymp n^{-\frac{r}{d-1}},
\end{equation*}
where  $(a)_+:=\max\{a,0\}$, $A(n)\asymp B(n) $ means that $A(n)\ll
B(n)$ and $A(n)\gg B(n)$, and $A(n)\ll B(n)$ means that there exists
a positive constant $c$ independent of $n$ such that $A(n)\leq
cB(n)$.\end{thm}
In order to prove the main result, we used the discretization method, which
is based on the reduction of the calculation of widths of classes of functions
to the computation of widths of finite dimensional sets. However, we need to
overcome much difficulty, the methods and computations are more technical and complicated. In fact, in the proof of upper estimates, a key lemma (see Lemma 3.1 in Section 3) is important during the process of discretization, while its proof differs from special case for group $G=\mathbf{Z}_2^d$. In the general group case we need to use the properties of doubling weight and
generalized H\"{o}lder's inequality to get it. While in the proof of lower estimates, some support set property does not hold for $h$-spherical Laplace-Beltrami operator, we need to apply the property of general reflection group to solve the problem.

Similarly, we can also consider entropy numbers of weighted Sobolev space $W_p^r(\omega^B_{\kappa,\mu})$
(see the definitions of $W_p^r(\omega^B_{\kappa,\mu})$ in Section 5)
on the unit ball $\bd=\{x\in \mathbb{R}^d:\|x\|\le1\}$, in which
the weight function takes the form
$$\omega^B_{\kappa,\mu}(x)=h_\kappa^2(x)(1-\|x\|^2 )^{\mu-1/2},\ x\in\mathbb{B}^d,$$
where $h_\kappa$ is a reflection invariant weight function on $\mathbb{R}^d$ and $\mu>$0.
The case $h_\kappa=1$ corresponds to the classical weight function
$w_{\mu}(x)=(1-\|x\|^2 )^{\mu-1/2}$. We have the following result:
\begin{thm} Let $r>d(\frac{1}{p}-\frac{1}{q})_{+}(2\gamma_\kappa+1)$,
$1\leq p,q\leq\infty$, then we have
\begin{equation*}e_n\big(BW_{p}^r(\omega^B_{\kappa,\mu}),\
L_q(\omega^B_{\kappa,\mu}) \asymp n^{-\frac{r}{d}}.
\end{equation*}
\end{thm}

Moreover, there is also
a close relation between the unit ball and the simplex
$\mathbb{T}^d=\{x\in\mathbb{R}^d:x_1\ge0,\dots,x_d\ge0,1-|x|\ge0\},$ where
$|x|=x_1+\cdots+x_d,$ which allows us to further extend the results to the weighted functions of Soblev space
$W_{p}^r(\omega_{\kappa,\mu}^{T})$ on $\mathbb{T}^d$ (see the definitions of $W_{p}^r(\omega_{\kappa,\mu}^{T})$ in Section 6), in which the weight functions take the form
$$\omega_{\kappa,\mu}^{T}(x)=h_k^2(\sqrt{x_1},\dots,\sqrt{x_d})(1-|x|)^{\mu-1/2}/\sqrt{x_1\cdots x_d},$$
where $\mu\ge1/2$ and $h_\kappa$ is a reflection invariant weight function
defined on $\mathbb{R}^d$ and $h_\kappa$ is even in each of its variables.
The case $h_\kappa(x)=\prod\limits_{i=1}^d|x_i|^{2\kappa_i}$ gives the classical weight
function on the simplex.
\begin{thm} Let $r>d(\frac{1}{p}-\frac{1}{q})_{+}(2\gamma_\kappa+1)$,
$1\leq p,q\leq\infty$, then we have
\begin{equation*}e_n\big(BW_{p}^r(\omega_{\kappa,\mu}^{T}),\
L_q(\omega_{\kappa,\mu}^{T}) \asymp n^{-\frac{r}{d}}.
\end{equation*}
\end{thm}
This paper is organized as follows.  Section 2 is devoted to giving
the preliminary knowledge about $h$-harmonic analysis and weighted polynomial
inequalities on the sphere.
In Section 3, we  obtain some lemmas related to discretization of the problem of
estimates of entropy numbers.  Finally, we prove Theorems 1.1 in Section 4.
In Section 5, we deduce the result for entropy numbers  on the unit ball from those on the sphere.
Finally, in Section 6, we deduce the result for entropy numbers on the
simplex from those on the ball.

\section{$h$-harmonic analysis and weighted polynomial inequalities on the sphere}

In the following, we consider the weighted $L_p$ best approximation with respect to the
measure $h_\kappa^2d\sigma$ on $\ss$, the theory of $h$-harmonic is necessary. As an
extension of spherical harmonics, the usual Laplace operator is replaced by a sum of square Dunkl
operators.

The Dunkl operators are defined by
$$\mathcal{D}_i f(x)=\partial_i f(x)+\sum_{v\in \mathcal{R}_+}\kappa_v\frac{f(x)-f(x\sigma_v)}{\langle x,v\rangle}
\langle v,e_i\rangle, i=1,\ldots,d,$$
where $e_1=(1,0,\ldots,0),\cdots,e_d=(0,\ldots,0,1).$ From the definition of $\mathcal{D}_i$ we can see that
they are the first order differential-difference operators, for more properties of $\mathcal{D}_i$
refer to \cite{DX1,DX2}.

The analogue of the Laplace operator, which is called $h$-Laplacian, is defined by
$$\Delta_h=\mathcal{D}_1^2+\cdots\mathcal{D}_d^2.$$

We denote by $\Pi_n^d$ the space of polynomials of degree at most $n$ on $\ss$, i.e.
the polynomials of degree at most $n$ restricted on $\ss$,
$\mathcal{P}_n^d$ the subspace of homogeneous polynomials
of degree $n$ in $d$ variables. An $h$-harmonic polynomial $Y$ of degree $n$
is a homogeneous polynomial $Y\in\mathcal{P}_n^d$ such that $\Delta_h Y=0.$

A spherical $h$-harmonic $Y_n$ of degree $n$ is a homogeneous polynomial
of degree $n$ restricted on $\ss$ and $\Delta_h Y_n=0$. We denote by $\mathcal{H}_n^d(h_\kappa^2)$
the space of all spherical $h$-harmonics of degree $n$ on $\ss$.
Furthermore, spherical $h$-harmonic polynomials of different degree
are orthogonal with respect to the inner product
$$\langle f,g\rangle_\kappa:=\frac{1}{a_d^\kappa}\int_{\ss} f(x)g(x)h_\kappa^2(x)d\sigma(x),$$
that is, for $n\neq m$, $\langle Y_n,Y_m\rangle_\kappa=0, \ Y_n\in\mathcal{H}_n^d(h_\kappa^2),\ Y_m\in\mathcal{H}_m^d(h_\kappa^2).$
We can follow from the standard Hilbert space theory
that $$L_2(h_\kappa^2)=\bigoplus_{n=0}^{\infty}\mathcal{H}_n^d(h_\kappa^2),\ \ \ \Pi_n^d=\bigoplus_{k=0}^{n}\mathcal{H}_k^d(h_\kappa^2).$$

It is well known that $\dim \Pi_n^d\asymp n^{d-1},\  \dim \mathcal{H}_n^d\asymp n^{d-2}.$

In terms of the polar coordinates $y=ry'$, $r=\|y\|$, the $h$-Laplacian operator takes
the form \cite{Xu4}
$$\Delta_h=\frac{\partial^2}{\partial r^2}+
\frac{2\lambda_\kappa+1}{r}\frac{\partial}{\partial r}+\frac{1}{r^2}\Delta_{h,0},$$
where $\lambda_\kappa:=\frac{d-2}{2}+\gamma_\kappa$, $\Delta_{h,0}$ is the Laplace-Beltrami operator on the sphere.

It is analogous with the usual harmonics that the spherical $h$-harmonics are
eigenfunctions of Laplace-Beltrami operator $\Delta_{h,0}$ on the sphere, that
is $$\Delta_{h,0}Y_n(x)=-(n(n+2\lambda_\kappa))Y_n(x),\ \ x\in\ss,\ \ Y_n\in\mathcal{H}_n^d(h\kappa^2),$$

Denote by $proj_n^\kappa:L_2(h_\kappa^2)\longrightarrow\mathcal{H}_n^d(h_\kappa^2)$
the orthogonal projection from $L_2(h_\kappa^2)$ onto $\mathcal{H}_n^d(h_\kappa^2)$,
which can be expressed as
$$proj_n^\kappa(f)(x)=\frac{1}{a_d^\kappa}\int_{\ss} f(y)P_n(h_\kappa^2;x,y)h_\kappa^2(y)d\sigma(y),$$
where $P_n(h_\kappa^2;x,y)$ is the reproducing kernel of $\mathcal{H}_n^d(h_\kappa^2).$
Moreover, we get that for $f\in L_2(h_\kappa^2),$
$f(x)=\sum_{n=0}^\infty proj_n^{\kappa}f(x)$
in $L_2(h_\kappa^2)$ norm.

We note that the kernel $P_n(h_\kappa^2;x,y)$ has a compact formula in terms
of the intertwining operator which acts between ordinary harmonics and $h$-harmonics
and encodes essentially information on the action of reflection group. The intertwining operator
$V_\kappa$ is a linear operator on the space of algebraic polynomials on $\mathbb{R}^d$
which satisfies
$$\mathcal{D}_iV_\kappa=V_\kappa\partial_i,\ 1\le i\le d,\ V_\kappa 1=1,\ V_\kappa \mathcal{P}_n\subset\mathcal{P}_n,\ n\in\mathbb{N}_0.$$
One important property of the intertwining operator is that it is positive (see \cite{Ro}),
that is ,$V_\kappa p\ge0$ if $p\ge0.$

For the general reflection group $G$, the explicit formula of
$P_n(h_\kappa^2;x,y)$ is given by (see \cite{Xu2})
$$P_n(h_\kappa^2;x,y)=
\frac{n+\lambda_\kappa}{\lambda_\kappa}V_\kappa[C_n^\lambda(\langle\cdot,y\rangle)](x).$$
In the special case $G=\mathbf{Z}_2^d$, the kernel $P_n(h_\kappa^2;x,y)$ has an explicit formula (see \cite{Du2,Xu1,Xu2})
$$P_n(h_\kappa^2;x,y)=\frac{n+\beta_\kappa}{\beta_\kappa}c_\kappa\int_{[-1,1]^d}
C_n^{\beta_\kappa}(x_1y_1t_1+\cdots+x_dy_dt_d)\prod_{i=1}^d(1+t_i)(1-t_i^2)^{\kappa_i-1}dt,$$
where $\beta_\kappa=\frac{d-2}{2}+|\kappa|,\ |\kappa|=\sum_{i=1}^d\kappa_i,
\ \ c_\kappa=c_{\kappa_1}\cdots c_{\kappa_d},\
c_\lambda=\frac{\Gamma(\lambda+1/2)}{\sqrt{\pi}\Gamma(\lambda)}$
and $C_n^\lambda$ denotes the Gegenbauer polynomial of degree $n$.

Given $r>0,$ we define the fractional order of Laplace-Beltrami operator
$(-\Delta_{h,0})^{\frac{r}{2}}$ on $\ss$ in a distribution sense by
$$(-\Delta_{h,0})^{\frac{r}{2}}f=
\sum_{n=0}^\infty (n(n+2\lambda_\kappa))^{r/2} proj_n^\kappa(f).$$
where $f$ is a distribution on $\ss$. We call $(-\Delta_{h,0})^{\frac{r}{2}}f$ the $r$-th
order of distribution $f$.

For $f\in L_1(h_\kappa^2) $, $r\in \mathbb{R}$, the Fourier series of $(-\Delta_{h,0})^{\frac{r}{2}}f$
can be written as
$$(-\Delta_{h,0})^{\frac{r}{2}}f=\sum_{n=1}^{\infty} (n(n+2\lambda_\kappa))^{r/2}proj_n^\kappa(f).$$.

Let $r>0, \ 1\le p\le\infty$, the Sobolev space $W_p^r(h_\kappa^2,\ss)$ is defined by
\begin{align*}W_p^r(h_\kappa^2):=&W_p^r(h_\kappa^2,\ss):=\Big\{f\in L_p(h_\kappa^2) :
\|f\|_{W_p^r(h_\kappa^2)}<\infty,\
\\&\exists \ g\in L_p(h_\kappa^2), such \ \ that \ \  g=(-\Delta_{h,0})^{\frac{r}{2}}f \Big\},\end{align*}
where $\|f\|_{W_p^r(h_\kappa^2)}:=\|f\|_{p,\kappa}+
\|(-\Delta_{h,0})^{\frac{r}{2}}f\|_{p,\kappa}<\infty$.
While the Sobolev class $BW_p^r(h_\kappa^2)$ is defined to be
the unit ball of $W_p^r(h_\kappa^2)$.

Let $1\le p\le\infty,\ n\in\mathbb{Z}_+$, the best approximation of
$f\in L_p(h_\kappa^2)$ is defined by
$$E_n(f)_{p,\kappa}:=\inf\{\|f-P\|_{p,\kappa}:P\in\Pi_n^d\}.$$

It is known that for $f\in W_p^r(h_\kappa^2),\ 1\le p\le\infty,$
\begin{equation}\label{2.1}E_n(f)_{p,\kappa}\ll n^{-r}\|(-\Delta_{h,0})^{\frac{r}{2}}f\|_{p,\kappa}.\end{equation}

Let $\eta$ be a $C^{\infty}$ function on $[0,\infty)$ satisfying
$\eta(t)=1$ for $0\le t\le 1$ and $\eta(t)=0$ if $t\geq 2$.
Now define a sequence of operator $\eta_n$ for $n\in\mathbb{N}$
by \begin{equation}\label{2.2}\eta_n f(x)=\int_{\ss}f(y)L_{n,\eta}^\kappa (x,y)h_\kappa^2(y)d\sigma(y),\end{equation}
where $L_{n,\eta}^\kappa (x,y)=\sum_{k=0}^\infty \eta(\frac{k}{n})P_k(h_\kappa^2;x,y).$

For $f\in L_{p,h_\kappa^2},\ 1\le p\le\infty,$ the operator $\eta_n$ shares the following properties (see \cite[Proposition 3.7]{Xu3}):

$(1)\ \ \eta_n f\in \Pi_{2n-1}^d, \ and \  \eta_n p=p \ \ for \ p\in\Pi_n^d;$

$(2)\ \ \|\eta_n f\|_{p,\kappa}\ll\|f\|_{p,\kappa},\ \ n\in\mathbb{N};$

$(3)\ \ \|f-\eta_n f\|_{p,\kappa}\ll E_n(f)_{p,\kappa},\ \ n\in\mathbb{N}.$

From the property (1) of $\eta_n$ we can see that for $f\in\Pi_n^d$
\begin{equation}\label{2.3}f(x)=\frac{1}{a_d^\kappa}\int_{\ss} f(y)L_{n,\eta}^\kappa (x,y)h_\kappa^2(y)d\sigma(y).\end{equation}

For $f\in L_p(h_\kappa^2),$ we define
$$A_0(f)=\eta_1(f),\ \ A_s(f)=\eta_{2^s}(f)-\eta_{2^{s-1}}(f),\ \ s\geq 1.$$

Then $A_s(f)\in\Pi_{2^{s+1}}^d$ and $f=\sum_{s=0}^\infty A_s (f)$ in
$L_p(h_\kappa^2)$ norm. Furthermore, it follows from \eqref{2.1} that
\begin{equation}\label{2.4}\|A_s(f)\|_{p,\kappa}\ll 2^{-sr}\|(-\Delta_{h,0})^{\frac{r}{2}}f\|_{p,\kappa}.\end{equation}

Denote by
$d(x,y)=\arccos \langle x,y\rangle$ the geodesic distance between two
points $x$ and $y$ on $\mathbb{S}^{d-1}$,  $c(x,r)$ the spherical
cap centered at $x\in\mathbb{S}^{d-1}$ with radius $r>0$,
i.e., $c(x,r)=\{y\in\mathbb{S}^{d-1}:d(x,y)\leq r\}$.

Given $\varepsilon>0,$ a subset $\Lambda\subset\ss$ is called $\varepsilon$-separable
if$$\min\limits_{x\neq x'\in\Lz}d(x,x')\geq\varepsilon,$$
furthermore, a  maximal $\varepsilon$-separable set $\Lambda$ is an
$\varepsilon$-separable set satisfying
$$\max_{x\in\mathbb{S}^{d-1}}\min_{u\in\Lz}d(x,u)<\varepsilon.$$

A weight function $w$ on $\ss$ is called a doubling weight if
there exists a constant $L>0$ such that for any $x\in\ss$ and $r>0$
\begin{equation}\label{2.5}\int_{c(x,2r)}w(x)d\sigma(x)\le L\int_{c(x,r)}w(x)d\sigma(x),\end{equation}
the least constant $L$ for which \eqref{2.5} holds is called the doubling
constant of $w$ and is denoted by $L_w$.

We write for a doubling weight $w$ and measurable subset $E$ of $\ss$, $$w(E)=\int_{E} w(x)d\sigma(x).$$
For a spherical cap $B=c(x,r)$, interating \eqref{2.5} shows that $w(2^mB)\le L_w^mw(B)=2^{m\log_2 L_w}w(B).$
We will use the symbol $s_w$ to denote a number in $[0,\log_2 L_w]$ such that
$$\sup_{B\subset\ss}\frac{w(2^mB)}{w(B)}\}\le C_{L_w}2^{ms_w},\ \ m=1,2,\cdots,$$
where $C_{L_w}$ is a constant depending only on $L_w$ and the supremum is taken over all spherical
caps $B\subset\ss.$

It is also known that for $x,y\in\ss$ and $n=0,1,\cdots,$
\begin{equation}\label{2.6}w(c(x,\frac{1}{n}))\le C_{L_w}(1+nd(x,y))^{s_w}w(c(y,\frac{1}{n})).\end{equation}

By the definition of doubling weight, it is easily seen that the weight function
$w(x)=h_\kappa^2(x)$ satisfies the doubling condition, and for $x\in\ss$ and $n\in\mathbb{N}$ (see \cite{DX1})
\begin{equation}\label{2.7}w(c(x,\frac{1}{n}))\asymp n^{-(d-1)}\prod_{v\in \mathcal{R}_+}(|\langle x,v\rangle|+1/n)^{2\kappa_v}
=n^{-(d-1)}\prod_{j=1}^m(|\langle x,v_j\rangle|+1/n)^{2\kappa_j}.\end{equation}

For more properties of the doubling weight and weighted polynomial inequalities see \cite{DX1,GT1,GT2}.

The proof of our main results is based on the
following positive  cubature formulae and Marcinkiewz-Zygmund
inequalities (see \cite{BD}, \cite{Dai}, \cite{MNW}, \cite{NPW}).

\noindent{\bf Theorem  A.} Let $w$ be a doubling weight on $\ss$, there
exists a positive constant $\varepsilon$ depending only
on $d$ and $s_w$, such that for any $\delta\in(0,\varepsilon),$ and
any maximal $\frac{\delta}{n}$ separable subset $\Lambda\subset\ss,$
there exists a sequence of positive numbers
$\lambda_\xi\asymp w(c(\xi,\frac{\delta}{n})), \ \xi\in\Lambda,$
for which the following
$$\int_{\mathbb{S}^{d-1}}f(y)w(y)d\sigma(y)=\sum\limits_{\xi\in\Lambda}\lambda_{\omega}f(\xi)$$
holds for $f\in\Pi_{n}^d$.
Moreover, if the above equality is exact for $f\in\Pi_{3n}^d,$ then for $1\le p\le\infty,$ and $f\in\Pi_{n}^d,$
$$\|f\|_{p,\kappa}\asymp\bigg\{\begin{array}{ll}\Big(\sum\limits_{\xi\in \Lambda}
\lambda_\xi|f(\xi)|^p\Big)^\frac{1}{p},\ \ \
 &\mbox{if}\ 1\le p<\infty,\\\max\limits_{\xi\in\Lambda}|f(\xi)|,
 \ \ \ &\mbox{if}\ p=\infty,\end{array}$$
where the constants of equivalence depend only on $d$ and $s_w$.

\section {The main lemmas and proofs}

Let $\kappa=(\kappa_1,\cdots,\kappa_m)\in\mathbb{R}_{+}^m,\ \
v=(v_1,\cdots,v_m)$ with $v_j\in\ss, \ \ 1\le j\le m.$ The weight function $w(x)=\prod_{j=1}^m|\langle x,v_j\rangle|^{2\kappa_j}$
satisfies the doubling condition on $\ss$. Furthermore, we conclude
from Theorem A that the corresponding results hold for $w(x)$. For simplicity
of notation, we use the same signs as stated in Theorem A. We have the following important
lemma:

\begin{lem} Let $\kappa=(\kappa_1,\cdots,\kappa_m)\in\mathbb{R}_{+}^m,\ \
v=(v_1,\cdots,v_m)$ with $v_j\in\ss, \ \ 1\le j\le m,$ and $\Lambda,\lambda_\omega$ be as in Theorem A.
Then for the weight function $w(x)=\prod_{j=1}^m|\langle x,v_j\rangle|^{2\kappa_j},$
there exists a
constant $\beta\in(0,\frac{1}{2|\kappa|}),$ such that $$\sum_{\xi\in\Lambda}\lambda_\xi^{-\beta}\ll n^{(d-1)(1+\beta)}.$$
\end{lem}
\begin{rem}
When the group $G=\mathbf{Z}_2^d$, Huang and Wang (\cite{HW}) used elementary polar coordinate
method to give the above estimate. However,
their method cannot adapt for the general finite refection group. Instead, we use
the properties of doubling weight and generalized H\"{o}lder inequality to get
it.
\end{rem}

In order to prove Lemma 3.1, we need the following two lemmas.
\begin{lem}\cite[Lemma 4.6]{Dai} Suppose that $\alpha$ is a fixed nonnegative number,
$n$ is a positive integer and $f$ is
a nonnegative function on $\ss$ satisfying $$f(x)\le C_1(1+nd(x,y))^{\alpha}f(y) \ \ for\ \ all\ \  x,y\in\ss.$$
Then for any $0 < p <\infty$, there exists a nonnegative spherical polynomial $g\in\Pi_n^d$ such
that
$$C^{-1}f(x)\le g(x)^p\le C f(x) \ \ for \ \ any \ \ x\in\ss,$$
where $C > 0$ depends only on $d, C_1,p \ and \ \alpha$.\end{lem}

For $n=1,2,\cdots,$ it is often convenient to work with an approximation $w_n$
of weight function $w$, which is defined by \begin{equation}\label{3.1}w_n(x)=n^{d-1}\int_{c(x,\frac1n)}w(y)d\sigma(y).\end{equation}

\begin{lem}\cite[Corollary 3.4.]{Dai} For $f\in\Pi_n^d$ and $0<p<\infty,$

$$C^{-1}\|f\|_{p,w_n}\le\|f\|_{p,w}\le C\|f\|_{p,w_n},$$
where $C >0$ depends only on $d, L$ and $p$ when $p$ is small.\end{lem}

Now we are in the position to the {\bf{proof of Lemma 3.1}}:
\begin{proof}
It is easy to check that each $w_n(x)$ as defined in \eqref{3.1} is again a doubling
weight. By the definition of $w_n(x)$ and \eqref{2.6}, for any $x,y\in\ss$
and $n=1,2,\cdots,$
\begin{equation}\label{3.2}w_n(x)\ll (1+nd(x,y))^{s_w}w_n(y).\end{equation}

We conclude from \eqref{3.2} and Lemma 3.3 that for some $\beta>0$, whose range will be decided
later, there exists a nonnegative spherical
polynpmial $g\in\Pi_n^d$ such that $$g(x)\asymp(w_n(x))^{-(\beta+1)}.$$

By the equivalent form above and Lemma 3.4 we can easily get that
\begin{equation}\label{3.3}\|g\|_{1,w}\asymp\|g\|_{1,w_n}\asymp\|w_n^{-\beta}\|_{1}\end{equation}

Thus, for some $\beta>0,$ we have
\begin{align}\sum_{\xi\in\Lambda}\lambda_\xi^{-\beta}
&\asymp\sum_{\xi\in\Lambda}(w(c(\xi,\frac{1}{n}))^{-\beta}
=\sum_{\xi\in\Lambda}n^{(d-1)\beta} (w_n(\xi))^{-\beta}\notag
\\&=n^{(d-1)(1+\beta)}\sum_{\xi\in\Lambda}
(w_n(\xi))^{-(\beta+1)}\int_{c(\xi,1/n)}w(x)d\sigma(x)\notag
\\&\asymp n^{(d-1)(1+\beta)}\sum_{\xi\in\Lambda}\int_{c(\xi,1/n)}(w_n(x))^{-(\beta+1)}w(x)d\sigma(x)
\notag\\&\asymp n^{(d-1)(1+\beta)}\int_{\ss}(w_n(x))^{-(\beta+1)}w(x)d\sigma(x)
\notag\\&\label{3.4}\asymp n^{(d-1)(1+\beta)}\int_{\ss}(w_n(x))^{-\beta}d\sigma(x),\end{align}
where in the two equalities, we used the definition of $w_n(x)$ in \eqref{3.1},
in the second inequality, we used the fact that $w_n(x)\asymp w_n(\xi)$
for $x\in c(\xi,1/n)$, in the last second inequality,
we used the property of $\frac{\delta}{n}$-maximal separable set $\Lambda$, and in the
last inequality, we used \eqref{3.3}.

It remains to show that the  integration in \eqref{3.4} is controlled by some constant
independent of $n$ and $d$.

We note from \eqref{2.7} and \eqref{3.1} that for each $n\in\mathbb{N},$
$$(w_n(x))^{-\beta}\asymp \prod_{i=1}^m \big(|\langle x,v_i\rangle|+n^{-1}\big)^{-2\kappa_i\beta},$$

Denote by $$\widetilde{w_i}(x)=\big(|\langle x,v_i\rangle|+n^{-1}\big)^{-2\kappa_i\beta},\ i=1,\dots,m,$$
and $r_i=\frac{\sum_{j=1}^m\kappa_j}{\kappa_i}=\frac{|\kappa|}{\kappa_i},$
then $\sum_{i=1}^m \frac{1}{r_i}=1.$

We continue our proof, the generalized H$\ddot{o}$lder's inequality shows that
\begin{align} \int_{\ss}(w_n(x))^{-\beta}d\sigma(x)
&\asymp\int_{\ss}\prod_{i=1}^m \big(|\langle x,v_i\rangle|+n^{-1}\big)^{-2\kappa_i\beta}d\sigma(x)\notag
\\&=\int_{\ss}\widetilde{w_1}(x),\cdots,\widetilde{w_m}(x)d\sigma(x)\notag
\\&\label{3.5}\le \|\widetilde{w_1}\|_{r_1}\cdots \|\widetilde{w_m}\|_{r_m}.\end{align}

Notice that \begin{align}\|\widetilde{w_i}\|_{r_i}
&=\big(\int_{\ss}\big|\big(|\langle x,v_i\rangle|
+n^{-1}\big)^{-2\kappa_i\beta}\big|^{\frac{|\kappa|}{\kappa_i}}d\sigma(x)\big)^{1/r_i}\notag\\
&\label{3.6}= \big(\int_{\ss}\big(|\langle x,v_i\rangle|+n^{-1}\big)^{-2|\kappa|\beta}d\sigma(x)\big)^{1/r_i}
.\end{align}

By \eqref{3.6} and the rotation invariance of Lebesgue measure $d\sigma(x),$
\begin{equation}\label{3.7}\|\widetilde{w_1}\|_{r_1}^{{r_1}}=\|\widetilde{w_2}\|_{r_2}^{{r_2}}=\cdots= \|\widetilde{w_m}\|_{r_m}^{{r_m}}\end{equation}

It follows from \eqref{3.5}-\eqref{3.7} that
\begin{align}&\int_{\ss}(w_n(x))^{-\beta}d\sigma(x)
\ll \|\widetilde{w_1}\|_{r_1}\cdots \|\widetilde{w_m}\|_{r_m}\notag
\\&\le\Big(\int_{\ss}\big(|\langle x,v_i\rangle|+n^{-1}\big)^{-2|\kappa|\beta}d\sigma(x)\Big)^{1/r_1+\cdots+1/r_m}\notag
\\&=\int_{\ss}\big(|\langle x,v_i\rangle|+n^{-1}\big)^{-2|\kappa|\beta}d\sigma(x)
=\omega_{d-1}\int_{-1}^1\frac{(1-t^2)^{\frac{d-3}{2}}}{(|t|+n^{-1})^{2|\kappa|\beta}}dt\notag
\\&=2\omega_{d-1}\int_{0}^1\frac{(1-t^2)^{\frac{d-3}{2}}}{(t+n^{-1})^{2|\kappa|\beta}}dt
\ll \sum_{l=1}^n\int_{\frac{l-1}{n}\le t\le \frac{l}{n}}\frac{dt}{(t+n^{-1})^{2|\kappa|\beta}}\notag
\\&\label{3.8}\le \sum_{l=1}^n (\frac{l}{n})^{-2|\kappa|\beta}\frac{1}{n}
=n^{2|\kappa|\beta-1}\sum_{l=1}^n l^{-2|\kappa|\beta}\ll 1,
\end{align}
where the last equality in \eqref{3.8} holds if $\beta\in(0,\frac{1}{2|\kappa|}).$
The proof of Lemma 3.1 is completed.\end{proof}

It is well known that $\#\Lambda\asymp n^{d-1}.$
Now for each integer $s$ and $w(x)=h_\kappa^2(x)$, given $\{w_{s,k}:k\in\Lambda_s^d\}$
of distinct points $w_{s,k}\in\ss$ satisfying
$$\min\limits_{k\neq k'\in\Lz_s^d}d(w_{s,k},w_{s,k'})\geq\frac{\delta}{2^{s+4}}\ \
\mbox{and}\ \
\max_{x\in\mathbb{S}^{d-1}}\min_{k\in\Lz_s^d}d(x,w_{s,k})<\frac{\delta}{2^{s+4}},$$

By Theorem A, there exists a sequence of numbers
\begin{equation}\label{3.9}\lambda_{s,k}\asymp\int_{c(w_{s,k},\frac{\delta}{2^{s+4}})}h_\kappa^2(x)d\sigma(x), \ k\in\Lambda_s^d,\end{equation}
such that for any $f\in\Pi_{2^{s+4}}^d,$

\begin{equation}\label{3.10}\frac{1}{\omega_d^\kappa}\int_{\mathbb{S}^{d-1}}f(y)h_\kappa^2(y) d\sigma(y)=\sum\limits_{k\in\Lambda_s^d}\lambda_{s,k}f(\omega_{s,k}),\end{equation}
and
\begin{equation}\label{3.11}\|f\|_{p,\kappa}\asymp\bigg\{\begin{array}{ll}\big(\sum\limits_{k\in\Lambda_s^d}
\lambda_{s,k}|f(\omega_{s,k})|^p\big)^\frac{1}{p},\ \ \
 &\mbox{if}\ 1\le p<\infty,\\\max\limits_{k\in\Lambda_s^d}|f(\omega_{s,k})|,
 \ \ \ &\mbox{if}\ p=\infty.\end{array}\end{equation}

For $w=(w_1,\ldots,w_m)\in\mathbb{R}^m,$ we define as usual
$$\|x\|_{{\ell}_{p,w}^{m}}=\(\sum_{i=1}^m|x_i|^p w_i\)^{1/p}$$ for $1\le p<\infty$
and $\|x\|_{{\ell}_\infty^m}=\max_{1\le i\le m}|x_i|$ for $p=\infty.$

We denote by ${\ell}_{p,w}^m$ the set of vectors $x\in\mathbb{R}^m$ equipped
with the norm $\|\cdot\|_{{\ell}_{p,w}^m},$ and $B{\ell}_{p,w}^m$ the unit ball of ${\ell}_{p,w}^m.$
In the case $w=(1,\dots,1),$ it returns to the standard instance and
we write ${\ell}_p^m,\ \|\cdot\|_{{\ell}_p^m},\ B{\ell}_p^m$ instead of ${\ell}_{p,w}^m,\ \|\cdot\|_{{\ell}_{p,w}^m},\ B{\ell}_{p,w}^m.$
\begin{lem} Let $r>0,\ 1\le p,q\le\infty.$ Then for $1\le p\le q\le\infty,$
there exists a constant $\beta\in(0,\frac{1}{2\gamma_\kappa}),$ such that
\begin{align}\label{3.12}&e_n\big(BW_p^r(h_\kappa^2),L_q(h_\kappa^2)\big)
\ll\sum_{s=0}^{\infty}2^{-s\big(r-(\frac{1}{p}-\frac{1}{q})(d-1)\big)}
\sum_{k=0}^{s+1}\big(\frac{\#\Lambda_s^d}{2^{(k-1)(d-1)}}\big)^{\frac{1}{\beta}(\frac{1}{p}-\frac{1}{q})}
e_{n_{s,k}}\big(B{\ell}_p^{m_{s,k}},{\ell}_q^{m_{s,k}}\big),\end{align}
where $\sum_{s=0}^\infty\sum_{k=0}^{s+1}n_{s,k}\le n,\ m_{s,1}=2,\ m_{s,k}\asymp2^{(k-1)(d-1)},\ 2\le k\le s, and\
m_{s,s+1}=\#\Lambda_s^d-2^{s(d-1)}.$
\end{lem}
The following lemma plays a key role in the proof of Lemma 3.5.
\begin{lem} Let
$w=(w_1,\dots,w_m)\in \Bbb R^m$ satisfying $w_i>0,\ 1\le i\le m$ and
let
\begin{equation}\label{2.11-0}
\sum_{j=1}^m w_j^{-\gamma}\le m,\ \ \ {\rm for\ \ some}\ \
\gamma>0.\end{equation} Then for $1\le p\le q\le \infty$, there exists $j_0\in\mathbb{N}$ such
that $2^{j_0}\le m< 2^{j_0+1}$ and
$$ e_n(B{\ell}_{p,w}^m, \ell _{q,w}^m)\le\sum_{k=1}^{j_0} \Big(\frac m{2^{k-1}}\Big)^{\frac 1\gamma(\frac 1 p-\frac
1q)}e_{n_k}(B{\ell}_p^{m_k},\ell_q^{m_k}),$$ where $m_1=2,m_k=2^{k-1},2\le k\le{j_0}-1,m_{j_0}=m-2^{{j_0}-1}$,
and $\sum_{k=1}^{j_0}n_k\leq n$.\end{lem}

\begin{proof}
Without loss of generality we assume that \begin{equation}\label{2.11-1}w_1\le w_2\le\dots\le w_m.\end{equation}
First we claim that\begin{equation}\label{2.11-2}e_n(B\ell_{p,w}^m, \ell _{q,w}^m)=e_n(B\ell_{p,v}^m, \ell _{q}^m),\end{equation}
where $v=(w_1^{1-\frac{p}{q}},\dots,w_m^{1-\frac{p}{q}})$.

In fact, for any $x=(x_1,\dots,x_m)\in B\ell_{p,w}^m,$ the mapping $U$ defined by
$$Ux=(x_1w_1^{\frac{1}{q}},\dots,x_1w_1^{\frac{1}{q}})$$
yields an isometry of $B\ell_{p,w}^m$ onto $B\ell_{p,v}^m$ and $U^{-1}$ yields an isometry of $\ell _{q}^m$
onto $\ell _{q,w}^m$. As a consequence of the definition of the entropy numbers and the properties of $U$ and $U^{-1}$ we obtain \eqref{2.11-2}.

Now for arbitrary $x\in \mathbb{R}^m,$ we set $$S_k x=(x_1,\dots,x_k,0,\dots,0),k=1,\dots,m$$
and  $$\delta_1x=S_2x, \ \delta_jx=S_{2^j}x-S_{2^{j-1}}x,\ 2\le j\le j_0-1, \ {\rm and}  \ \delta_{j_0}x=x-S_{2^{j_0-1}}.$$
Then $x=\sum\limits_{j=1}^{j_0}\delta_jx.$
It follows from \eqref{2.11-0} and \eqref{2.11-1} that for $1 \le j\le m$
$$jw_j^{-\gamma}\le\sum\limits_{i=1}^{j}w_i^{-\gamma}
\le\sum\limits_{i=1}^m w_i^{-\gamma}\le m,$$
which implies  \begin{equation}\label{2.11-3}w_{j}^{-1}\le(\frac{m}{j})^{\frac{1}{\gamma}}, \ 1 \le j\le m.\end{equation}

Define the set $$ A_k:=\bigg\{x\in \mathbb{R}^{m_k}: (\sum_{i=1}^{m_k}|x_i|^pv_{i+2^{k-1}})^{\frac{1}{p}}\le 1\bigg\},1\le k\le j_0. $$
We deduce from \eqref{2.11-3} that for any $x\in A_k$
\begin{align*}&\sum_{i=1}^{m_k}|x_i|^p=\sum_{i=1}^{m_k}|x_i|^pv_{i+2^{k-1}}w_{i+2^{k-1}}^{-(1-\frac{p}{q})}
\\&\le\sum_{i=1}^{m_k}|x_i|^pv_{i+2^{k-1}}(\frac{m}{2^{k-1}})^{\frac{1}{\gamma}(1-\frac pq)}
\le(\frac{m}{2^{k-1}})^{\frac{1}{\gamma}(1-\frac pq)}
.\end{align*}
By the above inequality we have \begin{equation}\label{2.11-4}A_k\subset (\frac{m}{2^{k-2}})^{\frac{1}{\gamma}(\frac1p-\frac1q)}B\ell_p^{m_k}.\end{equation}
Combine with \eqref{2.11-2}, \eqref{2.11-4} and the definition and the properties of the entropy numbers we can get the conclusion of the lemma.
\end{proof}
\begin{rem}If we choose suitable $n_k$, we can prove that for $0<p<q<\infty,$
$$e_n(B{\ell}_{p,w}^m, \ell _{q,w}^m)\le\Big(\frac mn\Big)^{\frac 1\gamma(\frac 1 p-\frac1q)},\ \ if \ \ 1\le n\le \log2m,$$
$$e_n(B{\ell}_{p,w}^m, \ell _{q,w}^m)\le\Big(\frac mn\Big)^{\frac 1\gamma(\frac 1 p-\frac1q)}n^{-(1/p-1/q)},\ \ if \ \ \log2m\le n\le 2m,$$
and$$e_n(B{\ell}_{p,w}^m, \ell _{q,w}^m)\le\Big(\frac mn\Big)^{\frac 1\gamma(\frac 1 p-\frac1q)}2^{-\frac n{8m}}m^{-(1/p-1/q)},\ \ if \ \  2m\le n.$$
\end{rem}

\begin{rem}
Let $\omega=(\omega_1,\cdot\cdot\cdot,\omega_n)\in\mathbb{R}^n$ satisfying $\omega_i>0,1\le i\le n$ and let
$$\sum_{j=1}^n\omega_j^{-\beta}\le n, \ \ for \ \ some \ \ \beta>0.$$
Then for $1\le m\le n$ and $1\le p\le q\le\infty,$
$$S_m(B{\ell}_{p,w}^m, \ell _{q,w}^m)\le2^{\frac1\beta(\frac1p-\frac1q)}\Big(\frac nm\Big)^{\frac1\beta(\frac1p-\frac1q)}S_{\frac m2}(B{\ell}_{p}^n, \ell _{q}^n),$$
where $S_m$ denotes one of the Kolmogorov $m$-width $d_m$ or the linear m-width $\delta_m$
(see \cite{Dai}) or the Gelfand width $d_m$ (see \cite{HW}).
We remark that the similar result cannot hold for the entropy numbers. This is
due to the fact that $S_m(B{\ell}_{p,w}^m, \ell _{q,w}^m),w)=0$ if $m\ge n$, while
$_m(B{\ell}_{p,w}^m, \ell _{q,w}^m),w)>0$ for all $n,m \in \mathbb{N}.$
\end{rem}
Now we are ready to prove {\bf{Lemma 3.5.}}

\

\noindent {\it Proof of Lemma 3.5}:

Denote by $id: X\mapsto Y$ the identity operator from $X$ to
$Y$, where $X$ and $Y$ are normed linear spaces. Then
$$e_n\big(BW_p^r(h_\kappa^2),L_q(h_\kappa^2)\big)=e_n\big(id:
W_p^r(h_\kappa^2)\mapsto L_q(h_\kappa^2)\big).$$

Since for $f\in L_q(h_\kappa^2),$ $f=\sum_{s=0}^\infty A_s f$ in
$L_q(h_\kappa^2)$ norm, which implies that $id=\sum_{s=0}^\infty A_s.$

It follows immediately that $$e_n\big(id:
W_p^r(h_\kappa^2)\mapsto L_q(h_\kappa^2)\big)\le
\sum_{s=0}^\infty e_{n_s}\big(A_s:W_p^r(h_\kappa^2)\mapsto L_q(h_\kappa^2)\big),$$
where $\sum_{s=0}^\infty n_s\le n.$

It is clearly from \eqref{2.4} that for $f\in W_p^r(h_\kappa^2),\ 1\le p\le\infty$
and $s\ge 0,$
$$\|A_s(f)\|_{p,\kappa}\ll2^{-sr}\|f\|_{W_p^r(h_\kappa^2)}.$$

Hence, \begin{align}&e_{n_s}\big(A_s:W_p^r(h_\kappa^2)\mapsto L_q(h_\kappa^2)\big)
=e_{n_s}\big(A_s(BW_p^r(h_\kappa^2)),L_q(h_\kappa^2)\big)\notag
\\&\ll2^{-sr}e_{n_s}\big(BL_p(h_\kappa^2)\cap\Pi_{2^{s+1}}^d, L_q(h_\kappa^2)\big)\notag\end{align}

In order to prove \eqref{3.12}, we proceed to show that
\begin{align}&\label{3.13}e_{n_s}\big(BL_p(h_\kappa^2)\cap\Pi_{2^{s+1}}^d, L_q(h_\kappa^2)\big)\notag
\\&\ll 2^{s(d-1)(\frac{1}{p}-\frac{1}{q})}
\sum_{k=1}^{s}\big(\frac{\#\Lambda_s^d}{2^{(k-1)(d-1)}}\big)^{\frac{1}{\beta}(\frac{1}{p}-\frac{1}{q})}
e_{n_{s,k}}(B{\ell}_p^{m_{s,k}},{\ell}_q^{m_{s,k}}).\end{align}

Now we
define the operators $U_s: BL_p(h_\kappa^2)\cap\Pi_{2^{s+1}}^d\mapsto
\ell_{p,w}^{\#\Lz_s^d}$ and $V_s:\ell_{q,w}^{\#\Lz_s^d}\mapsto
L_q(h_\kappa^2)$ by
$$U_s(f)=(f(\omega_{s,1}),\dots,f(\omega_{s,\#\Lz_s^d})),\ \ f\in BL_p(h_\kappa^2)\cap\Pi_{2^{s+1}}^d
,\ \ w=(\lambda_{s,k})_{k\in\Lambda_s^d}$$and
$$V_s(a)(x)=\sum_{k=1}^{\#\Lz_s^d}a_{s,k}\lz_{s,k}L_{2^{s+1},\eta}^\kappa(x,\omega_{s,k}),\
x\in\ss, \ a=(a_{s,k})\in \ell_q^{\#\Lz_s^d},$$
where $L_{n,\eta}^\kappa$ is defined as in \eqref{2.2},  $\#\Lz_s^d,\,\lz_{s,k},\,\omega_{s,k}$  as in \eqref{3.10}
and \eqref{3.11}.

We conclude from \eqref{3.11} that for $f\in BL_p(h_\kappa^2)\cap\Pi_{2^{s+1}}^d,$
\begin{equation}\label{3.14}\|U_s(f)\|_{\ell_{p,w}^{\#\Lz_s^d}}\asymp\|f\|_{p,\kappa}.\end{equation}

Furthermore, we claim that for $1\le p\le\infty,$
\begin{equation}\label{3.15}\|V_s(a)\|_{q,\kappa}\ll\|a\|_{\ell_{q,w}^{\#\Lz_s^d}}.\end{equation}

As a matter of fact, for $q=1$, \eqref{3.15} follows directly from
property (2) of $\eta_n$
$$\max_{y\in \ss}\|L_n^\kappa(\cdot,y)\|_{1,\kappa}\ll 1.$$
For $q=\infty$, by \eqref{3.11} with $p=1$ and $f(y)=L_n^\kappa(x,y)$, we have
$$
\|V_s(a)\|_{\infty}\le
\|a\|_{\ell_\infty^{\#\Lz_s^d}}\sum_{i=1}^{\#\Lz_s^d}\lambda_{s,k}|L_{2^{s+1},\eta}^\kappa(x,\omega_{s,k})|\asymp
\|a\|_{\ell_\infty^{\#\Lz_s^d}}\ \|L_{2^{s+1},\eta}(x,\cdot)\|_{1,\kappa}\ll
\|a\|_{\ell_\infty^{\#\Lz_s^d}}.$$
For $1<q<\infty$, \eqref{3.15} follows from the Riesz-Thorin theorem.

On account of \eqref{2.3} and \eqref{3.10}, we check at once that
for $f\in L_p(h_\kappa^2)\cap\Pi_{2^{s+1}}^d,$
$$f(x)=\frac{1}{a_d^\kappa}\int_{\ss} f(y)L_{2^{s+1},\eta}^\kappa (x,y)h_\kappa^2(y)d\sigma(y)
=\sum_{k\in\Lambda_s^d}\lambda_{s,k}f(\omega_{s,k})L_{2^{s+1},\eta}^\kappa (x,\omega_{s,k})
=V_sU_s(f)(x).$$

This implies that the operator $id:L_p(h_\kappa^2)\cap\Pi_{2^{s+1}}^d\mapsto
L_q(h_\kappa^2)$ can be factored as follows:
$$id: L_p(h_\kappa^2)\cap\Pi_{2^{s+1}}^d
\stackrel{U_s}{\longrightarrow} \ell_{p,\omega}^{\#\Lz_s^d}
\stackrel{id}{\longrightarrow}
\ell_{q,\omega}^{\#\Lz_s^d}\stackrel{V_s}{\longrightarrow}  L_q(h_\kappa^2).$$

It then
follows by \eqref{3.14}, \eqref{3.15} and the properties of entropy numbers that
\begin{align}
e_{n_s}(id:L_p(h_\kappa^2)\cap\Pi_{2^{s+1}}^d\mapsto L_q(h_\kappa^2))&\le\|V_s\|\,
e_{n_s}(id: \ell_{p,\omega}^{\#\Lz_s^d}\mapsto \ell_{q,w}^{\#\Lz_s^d})\,\|U_s\|\notag\\&\ll
e_{n_s}(B{\ell}_{p,w}^{\#\Lz_s^d},\ell_{q,w}^{\#\Lz_s^d})\label{3.16}.\end{align}

At last, it follows from \eqref{3.9} and Lemma 3.1 that
for a constant $\beta\in(0,\frac{1}{2\gamma_\kappa}),$
$$\sum_{k\in\Lambda_s^d}(\#\Lambda_s^d\lambda_{s,k})^{-\beta}\ll \#\Lambda_s^d.$$

Hence by the inequality above we use Lemma 3.6 to the vector $\tilde w=(c\#\Lambda_s^d \lambda_{s,k})_{k\in\Lambda_s^d}$, we get
that
\begin{align}&e_{n_s}(B{\ell}_{p,w}^{\#\Lz_s^d},\ell_{q,w}^{\#\Lz_s^d})
\asymp\ (\#\Lambda_s^d)^{\frac{1}{p}-\frac{1}{q}}e_{n_s}(B{\ell}_{p,\tilde w}^{\#\Lz_s^d},
\ell_{q,\tilde w}^{\#\Lz_s^d})\notag\notag\\&\ll2^{s(d-1)(\frac{1}{p}-\frac{1}{q})}
\sum_{k=0}^{s}\big(\frac{\#\Lambda_s^d}{2^{(k-1)(d-1)}}\big)^{\frac{1}{\beta}(\frac{1}{p}-\frac{1}{q})}
e_{n_{s,k}}(B{\ell}_p^{m_{s,k}},{\ell}_q^{m_{s,k}})\label{3.17},\end{align}
which combines with \eqref{3.16}, gives \eqref{3.13} and this finish the proof of Lemma 3.5.$\hfill\Box$

\begin{lem}
Let  $r>0,\ 1\le
 p,q\le\infty$.  Then  there exists
a positive integer $N$ such that  $N\asymp n$, $N\geq2n$, and
\begin{equation*}e_n(BW_p^r(h_\kappa^2),L_q(h_\kappa^2))\gg
n^{-\frac{r}{d-1}+\frac{1}{p}-\frac{1}{q}}e_n(B{\ell}_p^N
,\ \ell_q^N).\end{equation*}
\end{lem}
\begin{proof}
The proof idea is standard (see for example \cite{BD,HW,WWW}).
However, because of the support property of $h$-spherical Laplace-Beltrami operator, we need to overcome some difficulty.

For convenience, we denote again by
$h_\kappa(x)=\prod_{v\in\mathcal{R}_+}|\langle x,v\rangle|^{\kappa_v}
=\prod_{j=1}^m|\langle x,v_j\rangle|^{\kappa_j}.$
To prove the lower estimate, let
$$E_j=\{x\in\ss: |\frac{\pi}{2}-d(x,v_j)|\le2\varepsilon_{d,m}\},\ 1\le j\le m,$$
and
$$\widetilde{E}_j=\{x\in\ss: |\frac{\pi}{2}-d(x,v_j)|\le\varepsilon_{d,m}\},\ 1\le j\le m,$$
where $\varepsilon_{d,m}$ is a sufficiently small positive constant depending only on
$d$ and $m$.

Le $|E|$ denote the Lebesgue measure of a set $E\subset\ss$.
A straightforward calculation shows that
$$\big|\bigcup_{j=1}^mE_j\big|\le\sum_{j=1}^m|E_j|\le c_{d}m\varepsilon_{d,m}\le\frac12|\ss|$$
provided that $\varepsilon_{d,m}$ is small enough. If $x\in\ss\backslash(\bigcup_{j=1}^mE_j),$ then
$$|\langle x,v_j\rangle|\ge\sin2\varepsilon_{d,m},\ j=1,\dots,m.$$
Hence $$h_\kappa^2(x)=\prod_{j=1}^m|\langle
x,v_j\rangle|^{2\kappa_j}
\gg(\sin2\varepsilon_{d,m})^{2|\kappa|}\gg1,\
x\in\ss\backslash(\bigcup_{j=1}^mE_j).$$ We assume that
$l\in\mathbb{N}$ is sufficiently large and $c_1l^{d-1}\le n\le
c_2l^{d-1}$ with $c_1,c_2>0$ being independent of $n$ and $l$. We
let $\{x_i\}_{i=1}^N\subset\ss\backslash(\bigcup_{j=1}^mE_j)$ such
that $N\asymp l^{d-1}$,
$$c(x_i,\frac1l)\cap c(x_j,\frac1l)=\varnothing,\ \rm{if}\ i\neq j,$$
and $c(x_i,\frac1l)\subset\ss\backslash(\bigcup_{j=1}^m\widetilde{E}_j).$
Obviously, such points $\{x_i\}_{i=1}^N$ exist. We may take $c_2$ sufficiently small so that $N\ge2n.$

Let $\varphi$ be a nonnegative $C^{\infty}$-function on $\mathbb{R}$ supported in $[0,1]$ and
be equal to $1$ on $[0,\frac12].$ We define
$$\varphi_i(x)=\varphi (ld(x,x_i)),\ i=1,\dots,N$$
and set
$$A_N:=\Big\{f_a(x)=\sum_{i=1}^N a_i\varphi_i(x): a=(a_1,\dots,a_N)\in\mathbb{R}^N\Big\}.$$
It is clearly that $$\supp\varphi_i\subset
c(x_i,\frac1l)\subset\ss\backslash(\bigcup_{j=1}^m\widetilde
E_j)$$ and
 $$\|\varphi_i\|_{p,\kappa}\asymp\(\int_{c(x_i,\frac1l)}|\varphi (ld(x,x_i))|^p
d\sigma(x)\)^{1/p}\asymp l^{-\frac{d-1}{p}}$$
and $$\supp\varphi_i\bigcap\supp\varphi_j=\varnothing,\ (i\neq j).$$
Hence, for $f_a\in A_N,\ a=(a_1,\dots,a_N)\in\mathbb{R}^N$,
\begin{equation}\label{3.18}
\|f_a\|_{p,\kappa}\asymp\(l^{-(d-1)}\sum_{i=1}^N|a_i|^p\)^{1/p}=l^{-\frac{d-1}{p}}\|a\|_{{\ell}_p^N}.
\end{equation}
Next, we note that if $f\in C_c^{\infty}(\mathbb{R}^d),$ then by
the definition of $\mathcal{D}_k$ we have
$$\supp\mathcal{D}_k f\subset\bigcup_{\rho\in G}\rho(\supp f),\ k=1,\dots,d.$$
where $G$ is the finite reflection group generated by
$\mathcal{R}_+$, $\rho(E)=\{\rho x:x\in E\}.$ This, combining
with the fact that the set $\bigcup_{\rho\in G}\rho (E)$ is
invariant under the action of the group $G$,  means that for
$v\in\mathbb{N},$
$$\supp(-\Delta_h)^vf\subset\bigcup_{\rho\in G}\rho(\supp f).$$
By the definition of $h$-Laplacian-Beltrami operator, for $\xi\in\ss,$ we have
$$\Delta_{h,0}\varphi_i(\xi)=\Delta_h\(\varphi_i(md(\frac{x}{\|x\|},x_i))\){\Big|_{x=\xi}},$$
which implies that
\begin{equation}\label{3.18-0}
\supp\Delta_{h,0}\varphi_i\subset\bigcup_{\rho\in G}\rho(c(x_i,\frac1l)).
\end{equation}
Furthermore,
$$\supp(-\Delta_{h,0})^v\varphi_i\subset\bigcup_{\rho\in G}\rho(c(x_i,\frac1l)),\ v=1,2,\dots.$$
Finally, we can verify that
$$\|(-\Delta_h)^v\varphi_i\|_{\infty}\le l^{2v},\ \ 1\le i\le N,\ v=1,2,\dots,$$
which implies that
\begin{equation}\label{3.19}
\|(-\Delta_{h,0})^v\varphi_i\|_{\infty}\ll l^{2v},\ \ v=1,2,\dots.
\end{equation}
By \eqref{3.18-0} and \eqref{3.19}, we have
\begin{equation}\label{3.19-0}
\|(-\Delta_{h,0})^v\varphi_i\|_{p,\kappa}\ll l^{2v-(d-1)/p},\ \ v=1,2,\dots.
\end{equation}
For $f_a\in A_N,\
a=(a_1,\dots,a_N),$ Then
$$(-\Delta_{h,0})^vf_a(x)=\sum_{i=1}^Na_i(-\Delta_{h,0})^v\varphi_i(x).$$
Since $\supp(-\Delta_{h,0})^v\varphi_i\subset\bigcup_{\rho\in G}\rho(c(x_i,\frac1l)),\ v\in \mathbb{N}$, we get
 that for any  $x\in\ss$,
$$\#\{i : a_i(-\Delta_{h,0})^v\varphi_i(x)\neq 0, \ 1\le i\le N\}
\le \sum_{i=1}^N\chi_{\bigcup\limits_{\rho\in G}\rho(c(x_i,\frac1l))}(x), \ v\in \mathbb{N},$$  where
$\chi_E(x)$ is the characteristic function of a set $E$. We note
that
\begin{eqnarray}\label{3.20}
&&\sum_{i=1}^N\chi_{\bigcup\limits_{\rho\in G}\rho(c(x_i,\frac1l))}(x)
\le\sum_{i=1}^N\sum_{\rho\in G}\chi_{\rho(c(x_i,\frac1l))}(x)\notag
\\&&=\sum_{\rho\in G}\sum_{i=1}^N\chi_{\rho(c(x_i,\frac1l))}(x)
\le\sum_{\rho\in G}1=\# G,
\end{eqnarray}
where in the last inequality we used the pairwise disjoint
property of $\{\rho( c(x_i,\frac1l))\}_{i=1}^N$ for any $\rho\in G$.
We deduce from \eqref{3.19-0} and \eqref{3.20} that for $1\le
p\le\infty$,
\begin{align}
\|(-\Delta_{h,0})^vf_a\|_{p,\kappa}
&\ll\(\int_{\ss}|\sum_{i=1}^Na_i(-\Delta_{h,0})^v\varphi_i(x)|^pd\sigma(x)\)^{1/p}\notag\\
&\le
(\#G)^{1-1/p}\(\int_{\ss}\sum_{i=1}^N|a_i|^p|(-\Delta_{h,0})^v\varphi_i(x)|^pd\sigma(x)\)^{1/p}\notag\\
&\ll l^{2v-\frac{d-1}{p}}\|a\|_{{\ell}_p^N}, \ \ v\in\Bbb
N.\label{3.21}
\end{align}
It then follows by Kolmogorov type inequality (see \cite[Theorem
8.1]{Di}), \eqref{3.21} and \eqref{3.18} that for $v>r,\ v\in\Bbb
N$,
$$\|(-\Delta_{h,0})^{r/2}f_a\|_{p,\kappa}
\ll\|(-\Delta_{h,0})^vf_a\|_{p,\kappa}^{\frac{r}{2v}}\|f_a\|_{p,\kappa}^{\frac{2v-r}{2v}}
\ll l^{r-\frac{d-1}{p}}\|a\|_{{\ell}_p^N}\ll l^r\|f_a\|_{p,\kappa}.$$
Recall that $n\asymp l^{d-1}$, hence
\begin{equation}\label{3.29}
c n^{-\frac{r}{d-1}}(BL_p(h^2_{\kappa})\cap A_N)\subset BW_p^r(h^2_{\kappa})\cap A_N,
\end{equation}
where $c$ is a positive constant independent of
$n$. We have
\begin{align*}
e_n\big(BW_p^r(h^2_{\kappa}),L_q(h^2_{\kappa})\big)&\geq
e_n\big(BW_p^r(h^2_{\kappa})\cap A_N,L_q(h^2_{\kappa})\big)
\\&\gg n^{-\frac{r}{d-1}} e_n\big(BL_p(h^2_{\kappa})\cap A_N,L_q(h^2_{\kappa})\cap A_N\big)
\\&\gg n^{-\frac r{d-1}+\frac{1}{p}-\frac{1}{q}}e_n(B\ell_p^N,\ell_q^N).
\end{align*}
The proof of Lemma 3.9 is finished.
\end{proof}

\section{Proof of Theorems 1.1}

\noindent {\it Proofs of Theorems 1.1}

First we consider
the lower estimates. For  all $k$, $m\in\mathbb{N}$, we have  (see
\cite{ET, Sc, Ku}): for $0<p\leq q\leq\infty$
\begin{equation}\label{4.1}e_k(B{\ell}_p^m,\ell_q^m)\asymp\bigg\{\begin{array}{ll}1,\ \
\ &1\leq k<\log2m,\\ \big(\frac{\log(1+\frac mk)}{k}\big)^{
1/p-1/q}, &\log 2m\leq k\leq 2m,\hspace{20pt}
\\ 2^{-\frac{k}{2m}}m^{1/q-1/p}, \ \ \ &
2m\leq k,\end{array}\end{equation}and in case $0<q<p\leq\infty$, it holds
\begin{equation}\label{4.2}e_k(B{\ell}_p^m,\ell_q^m)\asymp 2^{-k/(2m)}m^{1/q-1/p}. \end{equation}  This implies that if $m\asymp k$, then for all
$0<p,q\le\infty$,
\begin{equation}\label{4.3}e_k(B{\ell}_p^m,\ell_q^m)\asymp
k^{1/q-1/p}. \end{equation} By Lemma 3.9 and \eqref{4.3},  we obtain
the lower estimates for $e_n\big(BW_p^r(h_\kappa^2),L_q(h_\kappa^2)\big)$.

The only point remaining concerns the upper estimates.
We only need to consider the case $1\le p\le q\le\infty,$ since
for $1\le q<p\le\infty,$ we have the relation $BW_p^r(h_\kappa^2)\subset BW_q^r(h_\kappa^2),$
which implies that
$e_n\big(BW_p^r(h_\kappa^2),L_q(h_\kappa^2))\le e_n(BW_q^r(h_\kappa^2),L_q(h_\kappa^2)\big).$

Under the condition stated above, it follows from the proof in Lemma 3.5 that
\begin{align} &e_n\big(BW_p^r(h_\kappa^2),L_q(h_\kappa^2)\big)
\le\sum_{s=0}^\infty e_{n_s}\big(A_s(BW_p^r(h_\kappa^2)),L_q(h_\kappa^2)\big).\label{4.4}
\end{align}
For the convenience of estimation, we may take sufficiently small positive
number $\rho$ and define
$$n_s:=\left\{
\begin{array}{ll}
\big[\#\Lambda_s^d\cdot2^{(1-\rho)(d-1)(J-s)}\big]&\mbox{if}\,\,0\leq s\leq J,\\
  \ \big[\#\Lambda_s^d\cdot2^{(1+\rho)(d-1)(J-s)}\big]&\mbox{if}\,\,s>J,\\
\end{array}
\right.$$
here $[x]$ is the largest integer less than the real
number $x$. Since
$$\sum\limits_{s=0}^\infty n_s\ll\sum_{0\leq s\leq
J}\#\Lambda_s^d\cdot2^{(1-\rho)(d-1)(J-s)}+\sum_{s>J}\#\Lambda_s^d\cdot2^{(1+\rho)(d-1)(J-s)}\ll2^{J(d-1)},$$
we can choose an integer $J$ such that  $\sum\limits_{s=0}^\infty
n_s\leq n$ and  $2^{J(d-1)}\asymp n$.

By Lemma 3.5 , for arbitrary
$\varepsilon>0,$ take $\beta=\frac{1}{2(\gamma_k\ + \ \varepsilon)}$, we have
\begin{align}\label{4.5}
&e_n\big(BW_p^r(h_\kappa^2),L_q(h_\kappa^2)\big)
\ll\sum_{s=0}^{\infty}2^{-s\big(r-(\frac{1}{p}-\frac{1}{q})(d-1)\big)}
\sum_{k=0}^{s}(\frac{\#\Lambda_s^d}{2^{k(d-1)}})^{(\frac{1}{p}-\frac{1}{q})\frac{1}{\beta}}
e_{n_{s,k}}(B{\ell}_p^{m_{s,k}},{\ell}_q^{m_{s,k}})
\\&\ll\sum_{s=0}^{\infty}2^{-s\big(r-(\frac{1}{p}-\frac{1}{q})(d-1)\big)}
\sum_{k=0}^{s}2^{(s-k)(d-1)(\frac{1}{p}-\frac{1}{q})\frac{1}{\beta}}
e_{n_{s,k}}(B{\ell}_p^{m_{s,k}},{\ell}_q^{m_{s,k}})\notag
\\&=\sum_{0\le s< J}+\sum_{J\le s\le \frac{1+\rho}{\rho}J}+\sum_{s>\frac{1+\rho}{\rho}J}
:=I_1+I_2+I_3\notag.\end{align}

For $0\le s< J$, $n_s=[\#\Lambda_s^d\cdot2^{(1-\rho)(d-1)(J-s)}]$,
consequently we define $$n_{s,k}=[2^{(1-\rho)(d-1)(J-k)}m_{s,k}],$$ where $m_{s,k}$ is
the the same as in Lemma 3.5.
A short computation shows that
$\sum_{k=1}^s n_{s,k}\leq2^{(1-\rho)(J-s)}\#\Lambda_s^d\le n_s$
and $n_{s,k}\ge 2m_{s,k}.$

Hence, by the third case in \eqref{4.1}
\begin{align}\label{4.6}
I_1&\ll\sum_{s=0}^J2^{-s\big(r-(\frac{1}{p}-\frac{1}{q})(d-1)\big)}
\sum_{k=0}^{s}2^{(s-k)(d-1)(\frac{1}{p}-\frac{1}{q})\frac{1}{\beta}}
2^{-\frac{n_{s,k}}{2m_{s,k}}}m_{s,k}^{-(\frac{1}{p}-\frac{1}{q})}
\\&\ll\sum_{s=0}^J2^{-s\big(r-(\frac{1}{p}-\frac{1}{q})(d-1)(\frac{1}{\beta}+1)\big)}
\sum_{k=0}^{s}2^{-k(d-1)(\frac{1}{p}-\frac{1}{q})\frac{1}{\beta}}
2^{-2^{(1-\rho)(d-1)(J-k)}}2^{-k(d-1)(\frac{1}{p}-\frac{1}{q})}\notag
\\&\ll\sum_{s=0}^J2^{-s(r-(d-1)(\frac{1}{p}-\frac{1}{q})(\frac{1}{\beta}+1))}
2^{-s(d-1)(\frac{1}{p}-\frac{1}{q})(\frac{1}{\beta}+1)}
2^{-2^{(1-\rho)(d-1)(J-s)}}\notag
\\&\ll2^{-Jr}\sum_{s=0}^J2^{(J-s)r}2^{-2^{(1-\rho)(d-1)(J-s)}}\notag
\\&\ll2^{-Jr}\ll n^{-\frac{r}{d-1}}\notag
,\end{align}
which completes estimate of the first part.

Now we are ready to estimate $I_2.$ For $J <s\le\frac{1+\rho}{\rho}J,$ $n_s=[\#\Lambda_s^d\cdot2^{(1+\rho)(d-1)(J-s)}]$.
There is no loss of generality in assuming $n_s\asymp2^{J_1(d-1)},$ and
$$n_{s,k}:=\left\{
\begin{array}{ll}
\big[2^{(1-\rho)(d-1)(J_1-k)}m_{s,k}\big]&\mbox{if}\,\,0\leq k\leq J_1,\\
  \ \big[2^{(1+\rho)(d-1)(J_1-k)}m_{s,k}\big]&\mbox{if}\,\,k>J_1.\\
\end{array}
\right.$$
It also shows that
$\sum_{k=0}^s n_{s,k}\leq\#\Lambda_s^d\cdot2^{(1+\rho)(d-1)(J-s)}\le n_s$.

For convenience,
denote by $J_0:=\frac{1+\rho}{\rho}J$, we have
\begin{align}\label{4.7}
I_2&\ll\sum_{s=J}^{J_0}2^{-s\big(r-(d-1)(\frac{1}{p}-\frac{1}{q})\big)}
\sum_{k=0}^{s}2^{(s-k)(d-1)(\frac{1}{p}-\frac{1}{q})\frac{1}{\beta}}
e_{n_{s,k}}(B{\ell}_p^{m_{s,k}},{\ell}_q^{m_{s,k}})
\\&=\sum_{s=J}^{J_0}2^{-s\big(r-(d-1)(\frac{1}{p}-\frac{1}{q})(\frac{1}{\beta}+1)\big)}
\Big\{\sum_{0\le k\le J_1}2^{-k(d-1)(\frac{1}{p}-\frac{1}{q})\frac{1}{\beta}}
e_{n_{s,k}}(B{\ell}_p^{m_{s,k}},{\ell}_q^{m_{s,k}})\notag
\\&+\sum_{J_1< k\le s}2^{-k(d-1)(\frac{1}{p}-\frac{1}{q})\frac{1}{\beta}}
e_{n_{s,k}}(B{\ell}_p^{m_{s,k}},{\ell}_q^{m_{s,k}})\Big\}\notag
\\&:=\sum_{s=J}^{J_0}2^{-s\big(r-(d-1)(\frac{1}{p}-\frac{1}{q})(\frac{1}{\beta}+1)\big)}\Big\{P_1+P_2\Big\}\notag.
\end{align}

In order to get the estimate of $I_2$, we need to compute $P_1$ and $P_2$.
From the third case in \eqref{4.1}, it follows
\begin{align}\label{4.8}P_1&\ll\sum_{k\le J_1}2^{-k(d-1)(\frac{1}{p}-\frac{1}{q})\frac{1}{\beta}}
2^{-\frac{n_{s,k}}{2m_{s,k}}}m_{s,k}^{-(\frac{1}{p}-\frac{1}{q})}
\\&\ll\sum_{k\le J_1}2^{-k(d-1)(\frac{1}{p}-\frac{1}{q})\frac{1}{\beta}}
2^{-2^{(1-\rho)(d-1)(J_1-k)}}2^{-k(d-1)(\frac{1}{p}-\frac{1}{q})}
\\&\ll2^{-J_1(d-1)(\frac{1}{p}-\frac{1}{q})(\frac{1}{\beta}+1)}\notag.
\end{align}

The term $P_2$ need to be handled in a more complicated way, it deduces that

\begin{align}
P_2&=\sum_{ J_1<k\le s}2^{-k(d-1)(\frac{1}{p}-\frac{1}{q})\frac{1}{\beta}}
e_{n_{s,k}}(B{\ell}_p^{m_k},{\ell}_q^{m_k})\notag
\\&=\sum_{ J_1<k\le \frac{1+\rho}{\rho}J_1}+\sum_{ \frac{1+\rho}{\rho}J_1<k\le s}
=:Q_1+Q_2.\label{4.9}
\end{align}

For the case $J_1<k\le \frac{1+\rho}{\rho}J_1,$ recall that $n_{s,k}=[2^{(1+\rho)(d-1)(J_1-k)}m_{s,k}\big]$. It follows from the second case of \eqref{4.1} that
\begin{align}\label{4.10}
Q_1&=\sum_{ J_1<k\le \frac{1+\rho}{\rho}J_1}2^{-k(d-1)(\frac{1}{p}-\frac{1}{q})\frac{1}{\beta}}
e_{n_{s,k}}(B{\ell}_p^{m_k},{\ell}_q^{m_k})
\\&\ll\sum_{J_1< k\le {\frac{1+\rho}{\rho}J_1}}2^{-k(d-1)(\frac{1}{p}-\frac{1}{q})\frac1\beta}
\Big\{(1+\rho)(d-1)(k-J_1)\Big\}^{\frac{1}{p}-\frac{1}{q}}\notag\\&\cdot2^{(1+\rho)(d-1)(\frac{1}{p}-\frac{1}{q})(k-J_1)}
2^{-k(d-1)(\frac{1}{p}-\frac{1}{q})}\notag\\
& \ll \sum_{k\geq J_1}2^{-k(d-1)(\frac{1}{p}-\frac{1}{q})(\frac1\beta+1)}
\Big\{(d-1)(k-J_1)\Big\}^{1/p-1/q}2^{(1+\rho)(d-1)(\frac{1}{p}-\frac{1}{q})(k-J_1)}\notag\\
&\ll 2^{-J_1(d-1)(\frac{1}{p}-\frac{1}{q})(\frac1\beta+1)}\notag.
\end{align}

For the other case $\frac{1+\rho}{\rho}J_1<k\le s,$ it follows from the first case of \eqref{4.1} that
\begin{align}\label{4.11}
Q_2&=\sum_{ \frac{1+\rho}{\rho}J_1<k\le s}2^{-k(d-1)(\frac{1}{p}-\frac{1}{q})\frac{1}{\beta}}
e_{n_{s,k}}(B{\ell}_p^{m_k},{\ell}_q^{m_k})
\\&\ll2^{-\frac{1+\rho}{\rho}J_1(d-1)(\frac{1}{p}-\frac1q)\frac1\beta}
\ll2^{-J_1(d-1)(\frac{1}{p}-\frac{1}{q})(\frac1\beta+1)}\notag.
\end{align}

Combining with \eqref{4.9}-\eqref{4.11} we can get that
\begin{align}P_2\ll2^{-J_1(d-1)(\frac{1}{p}-\frac{1}{q})(\frac{1}{\beta}+1)}\label{4.12}\end{align}

By \eqref{4.7}, and the estimates of \eqref{4.8}, \eqref{4.12} and the definition of $J_1$, we finally get that
\begin{align} I_2&\ll\sum_{s=J}^{J_0}2^{-s\big(r-(d-1)(\frac{1}{p}-\frac{1}{q})(\frac{1}{\beta}+1)\big)}
2^{-J_1(d-1)(\frac{1}{p}-\frac{1}{q})(\frac{1}{\beta}+1)}\notag
\\ &\ll\sum_{s\ge J} 2^{-sr}2^{-(J-s)(1+\rho)(d-1)(\frac{1}{p}-\frac{1}{q})(\frac{1}{\beta}+1)}\notag
\\&\ll2^{-Jr}\ll n^{-\frac{r}{d-1}},\label{4.13}
\end{align}
where the last second inequality hods if \begin{align}r>(1+\rho)(d-1)(\frac{1}{p}-\frac{1}{q})(\frac{1}{\beta}+1)\label{4.14}\end{align}
for arbitrary $\varepsilon>0$ and small $\rho.$

We are left with the task of estimating $I_3$,
it follows from the property of entropy numbers and the definition of $J_0$ that
\begin{align}I_3&\ll\sum_{s\ge J_0}2^{-s\big(r-(\frac{1}{p}-\frac{1}{q})(d-1)\big)}
\sum_{k=0}^{s}2^{(s-k)(d-1)(\frac{1}{p}-\frac{1}{q})\frac{1}{\beta}}\notag\\
&\ll\sum_{s\ge J_0}2^{-s\big(r-(d-1)(\frac{1}{p}-\frac{1}{q})(\frac{1}{\beta}+1)\big)}\notag
\\&\ll2^{-J_0(r-(\frac{1}{p}-\frac{1}{q})(d-1)(\frac{1}{\beta}+1))}
\ll2^{-Jr}\ll n^{-\frac{r}{d-1}},\label{4.15}\end{align}
where in the last second inequality, we once again used \eqref{4.14}.

By the inequalities \eqref{4.5}-\eqref{4.6} and \eqref{4.13}, \eqref{4.15}, we
can deduce the upper estimate, which completes the proof.$\hfill\Box$

\section{Entropy numbers on the unit ball}
Now we consider the analogous result of entropy numbers of weighted Sobolev spaces on the unit ball $\bd$.

As introduced in the first section, take the weight function of the form
$$\omega^B_{\kappa,\mu}(x)=h_\kappa^2(x)(1-\|x\|^2 )^{\mu-1/2},\ x\in\mathbb{B}^d,$$
where $\mu>$0, $h_\kappa$ is a reflection invariant weight function on $\mathbb{R}^d$.

Denote by $L_p(\omega^B_{\kappa,\mu}),\ 1\le p<\infty$ the space of measurable functions defined on $\bd$
 with the finite norm
$$\|f\|_{p,\omega^B_{\kappa,\mu}}:=\bigg(\frac{1}{b_d^\kappa}\int_{\bd}|f(x)|^p \omega^B_{\kappa,\mu}d\sigma(x)\bigg)^{1/p}<\infty,$$
where $b_d^\kappa=\int_{\bd}\omega^B_{\kappa,\mu}dx$ is the normalization constant.
For $p=\infty$, we assume that $L_\infty$ is replaced by $C(\bd)$, the space of continuous function
on $\bd$ with the usual norm $\|\cdot\|_\infty$.

Let $\mathcal{V}_n^d(\omega^B_{\kappa,\mu})$ denote the space of orthogonal polynomials of degree $n$
 with respect to $\omega^B_{\kappa,\mu}$ on  $\bd$. There is a close relation between $h$-harmonics on the
 sphere $\mathbb{S}^{d}$ and orthogonal polynomials on the unit ball $\bd$(see \cite{Xu3} and reference in there). Denote by
$$h_{\kappa,u}(x_1,\cdots,x_{d+1})=h_{\kappa}(x_1,\cdots,x_{d})|x_{d+1}|^{\mu},\ (x_1,\cdots,x_{d+1})\in\mathbb{R}^{d+1}.$$
We can construct an one-to-one correspondence between the $h$-harmonics space $\mathcal{H}_n^d(h_{\kappa,u})$
 with respect to the function $h_{\kappa,u}$ and orthogonal polynomial space $\mathcal{V}_n^d(\omega^B_{\kappa,\mu})$. That is , for an even $Y_n\in\mathcal{H}_n^d(h_{\kappa,u})$ satisfying  $Y_n(x,x_{d+1})=Y_n(x,-x_{d+1}),$ we can write
$$Y_n(y)=r^nP_n(x), \ \ y=r(x,x_{d+1})\in\mathbb{R}^{d+1}, \ \ r=\|y\|, \ \ (x,x_{d+1})\in\Bbb {S}^{d}$$
in polar coordinates. Then $P_n$  is in $\mathcal{V}_n^d(\omega^B_{\kappa,\mu})$.

Moreover, by the elementary integral relation
$$\int_{\mathbb{S}^d}f(y)d\sigma(y)=\int_{\bd}\big\{f(x,\sqrt{1-\|x\|^2})+f(x,-\sqrt{1-\|x\|^2})\big\}dx$$
it follows that for $Y_n\in\mathcal{H}_n^d(h_{\kappa,u})$ and associated $P_n\in\mathcal{V}_n^d(\omega^B_{\kappa,\mu})$
$$\int_{\mathbb{S}^d}Y_n(y)h_{\kappa,u}(y)d\sigma(y)
=\int_{\bd}\big\{P_n(x,\sqrt{1-\|x\|^2})+P_n(x,-\sqrt{1-\|x\|^2})\big\}\omega^B_{\kappa,\mu}(x)dx.$$

Let $\Delta_h^{\kappa,\mu}$ denote the $h$-Laplacian with respect to $h_{\kappa,u}$ and $\Delta_{h,0}^{\kappa,\mu}$ the corresponding spherical $h$-Laplacian. For $y=r(x,x_{d+1}),\ \ (x,x_{d+1})\in\mathbb{S}^{d}$,
the spherical $h$-Laplacian can be written as(\cite{Xu5})
$$\Delta_{h,0}^{\kappa,\mu}=
\Delta_h-\langle x,\nabla\rangle^2-2\lambda\langle x,\nabla\rangle, \ \ \lambda=\sum_{v\in \mathcal{R}_+}\kappa_v+\mu+\frac{d-1}{2},$$
where $\Delta_h$ is the $h$-Laplacian associated with $h_\kappa$ on $\mathbb{R}^d.$
Define
$$D_{\kappa,\mu}^B:=\Delta_h-\langle x,\nabla\rangle^2-2\lambda\langle x,\nabla\rangle.$$
It follows that the elements of orthogonal polynomial subspace $\mathcal{V}_n^d(\omega^B_{\kappa,\mu})$ are
just the eigenfunctions of $D_{\kappa,\mu}^B$, that is
$$D_{\kappa,\mu}^B(P)=-n(n+2\lambda)P,\ \ P\in\mathcal{V}_n^d(\omega^B_{\kappa,\mu}).$$
For $f\in L_2(\omega^B_{\kappa,\mu}),$ the orthogonal expansion can be represented as
$$L_2(\omega^B_{\kappa,\mu})=\sum_{k=0}^\infty\bigoplus\mathcal{V}_n^d(\omega^B_{\kappa,\mu}),\ \
f=\sum_{k=0}^\infty proj_k^{\kappa,\mu}f,$$
where $proj_n^{\kappa,\mu}:L_2(\omega^B_{\kappa,\mu})\mapsto\mathcal{V}_n^d(\omega^B_{\kappa,\mu})$
is the projection operator.
The fractional power of $D_{\kappa,\mu}^B$  is defined by
$$(D_{\kappa,\mu}^B)^{r/2}(f)\sim(n(n+2\lambda))^{r/2}Pproj_n^{\kappa,\mu}(f),\ \ f\in L_p(\omega^B_{\kappa,\mu}).$$
Let $r>0, \ 1\le p\le\infty$, the Sobolev space $W_p^r(\omega^B_{\kappa,\mu},\mathbb{S}^d)$ is defined by
\begin{align*}W_p^r(\omega^B_{\kappa,\mu}):=&W_p^r(\omega^B_{\kappa,\mu},\mathbb{S}^d):=\big\{f\in L_p(\omega^B_{\kappa,\mu}) :
\|f\|_{W_p^r(\omega^B_{\kappa,\mu})}<\infty,\
\\&\exists \ g\in L_p(\omega^B_{\kappa,\mu}), such \ \ that \ \  g=(-D_{\kappa,\mu}^B)^{r/2}f \big\},\end{align*}
where $\|f\|_{W_p^r(\omega^B_{\kappa,\mu})}:=\|f\|_{p,\omega^B_{\kappa,\mu}}+
\|(-D_{\kappa,\mu}^B)^{r/2}f\|_{p,\omega^B_{\kappa,\mu}}$.
The Sobolev class $BW_p^r(\omega^B_{\kappa,\mu})$ is defined to be
the unit ball of $W_p^r(\omega^B_{\kappa,\mu})$.

For $f\in W_p^r(\omega^B_{\kappa,\mu},\mathbb{S}^d),$ define $T(f)(y)=f(x),\ \ y=(x,y_{d+1})\in\mathbb{S}^d,$
it follows from \cite[Proposition 4.1]{Xu3} that
$$\|(-D_{\kappa,\mu}^B)^{r/2}f\|_{p,\omega^B_{\kappa,\mu}}
=\|(-\Delta_{h,0}^{\kappa,\mu})^{r/2}T(f)\|_{p,\kappa,\mu}.$$

Under nearly the same argument, we can get exact order of the entropy numbers
$e_n\big(BW_{p}^r(\omega^B_{\kappa,\mu}),
L_q(\omega^B_{\kappa,\mu})\big)$  of Soblev classes  $BW_p^r(\omega^B_{\kappa,\mu})$ in $L_q(\omega^B_{\kappa,\mu})$ as stated in Theorem 1.2.
\section{Entropy numbers on the simplex}

Now we consider the problems of entropy numbers on the simplex $T^d$ with respect to the weight
$$\omega_{\kappa,\mu}^{T}(x)=h_k^2(\sqrt{x_1},\dots,\sqrt{x_d})(1-|x|)^{\mu-1/2}/\sqrt{x_1\cdots x_d},$$
where $\mu\ge1/2$ and $h_\kappa$ is a reflection invariant weight function
defined on $\mathbb{R}^d$ and $h_\kappa$ is even in each of its variables.
The last requirement essentially limits the weight functions to the case of group $\mathbb{Z}^d_2,$
for which
$$\omega_{\kappa}^T(x)=x_1^{\kappa_1-1/2}\cdot\cdot\cdot x_d^{\kappa_1-1/2}(1-|x|)^{\kappa_{d+1}-1/2}$$
which is the classical weight function on $T^d$.

For the case on the simplex $T^d$ with respect to the weight $\omega_{\kappa,\mu}^{T}(x)$, the related results
of entropy numbers can be deduced from the related results on the unit ball $\mathbb{B}^d$
with respect to the weight $\omega^B_{\kappa,\mu}(x)$.

The background on orthogonal expansion and approximation on $T^d$ is similar to the
case of the unit ball $\mathbb{B}^d$. The definitions of various notions, such as
$\|\cdot\|_{\omega_{\kappa,\mu}^{T}},\mathcal{V}_n(\omega_{\kappa,\mu}^{T})$,
are exactly the same as in the previous section with $T^d$ in place of $\mathbb{B}^d$.
There is a close relation between orthogonal polynomials on $\mathbb{B}^d$ and those on $T^d$.
Let $P_{2n}$ be an element of $\mathcal{V}_n(\omega_{\kappa,\mu}^B)$
and assume that $P_{2n}$ is even in each of its variables. Then we can write $P_{2n}$ as
$P_{2n}(x) = R_n(x_1^2,\cdot\cdot\cdot,x_d^2)$. It turns out that $R_n$ is
an element of $\mathcal{V}_n(\omega_{\kappa,\mu}^{T})$ and the relation is a one-to-one correspondence. In particular,
applying $D_{\kappa,\mu}^B$ on $P_{2n}$ leads to a second-order differential-difference operator acting on
$R_n$. Denote this operator by $D_{\kappa,\mu}^T.$ Then
$$D_{\kappa,\mu}^T(P)=-n(n+\lambda)P,\ \ P\in\mathcal{V}_n^d(\omega^T_{\kappa,\mu}),\ \ \lambda=\gamma_{\kappa}+\frac{d-1}2.$$

Denote by $L_p(\omega^T_{\kappa,\mu}),\ 1\le p<\infty$ the space of measurable functions defined on $T^d$
 with the finite norm
$$\|f\|_{p,\omega^T_{\kappa,\mu}}:=\bigg(\frac{1}{c_d^\kappa}\int_{T^d}|f(x)|^p \omega^B_{\kappa,\mu}dx\bigg)^{1/p}<\infty,$$
where $c_d^\kappa=\int_{T^d}\omega^B_{\kappa,\mu}dx$ is the normalization constant.
For $p=\infty$, we assume that $L_\infty$ is replaced by $C(T^d)$, the space of continuous function
on $T^d$ with the usual norm $\|\cdot\|_\infty$.

Let $r>0, \ 1\le p\le\infty$, under the standard argument, we can introduce the Sobolev spaces $W_p^r(\omega^T_{\kappa,\mu},T^d)$ with the norm

$$\|f\|_{W_p^r(\omega^T_{\kappa,\mu})}:=\|f\|_{p,\omega^T_{\kappa,\mu}}+
\|(-D_{\kappa,\mu}^T)^{r/2}f\|_{p,\omega^T_{\kappa,\mu}}.$$
The Sobolev class $BW_p^r(\omega^T_{\kappa,\mu})$ is defined to be
the unit ball of $W_p^r(\omega^T_{\kappa,\mu})$.

The related results for the entropy numbers of weighted Sobolev classes $BW_p^r(\omega^T_{\kappa,\mu})$ in $L_q(\omega^T_{\kappa,\mu})$ are obtained
due to the corresponding results on the unit ball $\mathbb{B}^d$.

\end{document}